\documentclass[11pt]{amsart}
\usepackage{amsmath,amssymb, graphicx, amscd,latexsym,here}
\makeatletter
\newtheorem{Theorem}{Theorem}
\newtheorem{Lemma}[Theorem]{Lemma}
\newtheorem{Corollary}[Theorem]{Corollary}
\newtheorem{Proposition}[Theorem]{Proposition}

\newtheorem{Remark}[Theorem]{Remark}

\newcommand{\eps}{\varepsilon}

\newcommand\vphi{\varphi}

\newcommand\al{\alpha}

\newcommand\be{\beta}

\newcommand\de{\delta}

\newcommand\cM{\mathcal  M}

\newcommand\BC{ {\mathbb C}}

\newcommand\BQ{ {\mathbb  Q}}
\newcommand\BZ{{\mathbb  Z}}
\newcommand\BR{ {\mathbb  R}}
\newcommand\BP{ {\mathbb  P}}

\newcommand\bfu{\mbox {\bf  u}}

\newcommand\bfw{\mbox {\bf  w}}

\newcommand\bfz{\mbox {\bf  z}}

\newcommand\Ker{\rm{Ker}\/}

\newcommand\id{\rm{id}}

\newcommand\lcm{\rm{lcm}\/}

\newcommand\inv{^{-1}}

\def\mapright#1{\smash{\mathop{\longrightarrow}\limits^{{#1}}}}

\def\inv{^{-1}}

\begin{document}
\title[On mixed plane curves of polar degree 1
]
{On mixed plane curves of polar degree 1
}

\author
[M. Oka ]
{Mutsuo Oka }
\address{\vtop{
\hbox{Department of Mathematics}
\hbox{Tokyo  University of Science}
\hbox{26 Wakamiya-cho, Shinjuku-ku}
\hbox{Tokyo 162-8601}
\hbox{\rm{E-mail}: {\rm oka@rs.kagu.tus.ac.jp}}
}}

\begin{abstract}
Let $f(\bfz,\bar\bfz)$ be a mixed strongly  polar homogeneous polynomial
of $3$ variables $\bfz=(z_1,z_2, z_3)$. It
 defines a Riemann surface 
$V:=\{[\bfz]\in \BP^{2}\,|\,f(\bfz,\bar\bfz)=0 \}$
in the complex projective space
$\BP^{2}$.
We will show that for an arbitrary given $g\ge 0$, there exists a
mixed polar homogeneous polynomial with polar degree 1 which defines a
 projective surface of genus $g$. For the construction, we introduce 
a new type of weighted homogeneous polynomials which we call
{\em polar weighted homogeneous polynomials of twisted join type}.
\end{abstract}
\maketitle

\maketitle

\section{Introduction}
Let $f(\bfz,\bar\bfz)$ be a strongly polar   homogeneous mixed polynomial 
of $n$-variables  $\bfz=(z_1,\dots, z_n)\in \BC^n$ 
with polar degree $q$ and radial
degree
$d$.
Recall that a strongly polar  homogeneous polynomial
$f(\bfz,\bar\bfz)$
satisfies the equality (\cite{MC}):
\begin{eqnarray}\label{action}
 f((t,\rho)\circ\bfz,\overline{(t,\rho)\circ\bfz})=t^{d}\rho^{q}f(\bfz,\bar\bfz),\quad
  (t,\rho)\in \BR^+\times S^1.
\end{eqnarray}
Here $(t,\rho)\circ\bfz$ is defined by the usual action
$(t,\rho)\circ\bfz\,=\,(t\rho z_1,\dots,t\rho z_n)$.
Let $\tilde V$ be the mixed affine hypersurface
\[
 \tilde V=f\inv(0)=\{\bfz\in \BC^n\,|\, f(\bfz,\bar\bfz)=0\}.
\]
We assume that $\tilde V$ has an isolated singularity at the origin. 
Let $f:\BC^n\setminus \tilde V\to \BC^*$  be the global 
Milnor fibration defined by $f$
and let $F$ be the fiber.
Namely $F$ is the hypersurface $f\inv(1)\subset \BC^n$.
The monodromy map $h:F\to F$  is defined by
\[
 h(\bfz)=(\eta z_1,\dots,\eta z_n),
\quad \eta=\exp(\frac{2\pi\,i}{q}). 
\]
We consider the smooth  projective hypersurface $V$ defined by
\[
 V=\{[\bfz]\in \BP^{n-1}\,|\, f(\bfz,\bar\bfz)=0\}.
\]

By (\ref{action}),
if  $\bfz\in f\inv(0)$  and $\bfz'$ is in the  same $\BR^+\times S^1$ orbit of $\bfz$, then $\bfz'\in f\inv(0)$.
Thus the hypersurface $V=\{[\bfz]\in \BP^{n-1}|f(\bfz)=0\}$ is well-defined.
Consider the quotient map
 $\pi: \BC^n\setminus\{O\}\to \BP^{n-1}$
and its restriction to the Milnor fiber
$\pi:\,F\to \BP^{n-1}\setminus V$.
This is a $q$-cyclic covering map.
In the previous paper, we have shown  that 
\begin{Theorem}{\rm (Theorem 11, \cite{MC})}
The embedding degree of $V$ is equal to the polar degree $q$.
\end{Theorem}

First we observe that 
\begin{Proposition}
The Euler characteristics satisfy the following equalities.
 \begin{enumerate}\label{formula1}
\item
$\chi(F)=q\,\chi(\BP^{n-1}\setminus V)$.
     \item  $\chi(\BP^{n-1}\setminus V)=n-\chi(V)$ and 
     $\chi(V)=n-\chi(F)/q$.
\item The following sequence is exact.
$$1\to \pi_1(F)\mapright{\pi_{\sharp}} \pi_1(\BP^{n-1}\setminus V)\to
     \BZ/q\BZ\to 1.$$
 \end{enumerate}
\end{Proposition}
\begin{Corollary}
If $q=1$, the projection $\pi:F\to\BP^{n-1}\setminus V$ is a
 diffeomorphism.
\end{Corollary}
\begin{Corollary}\label{genus formula1}
Assume that $n=3$. Then the genus $g(V)$ of $V$ is given by the formula:
\[
 g(V)=\frac 12\left(\frac{\chi(F)}{q}-1
\right)
\]
\end{Corollary}
The monodromy map $h:F\to F$ gives free $\BZ/q\BZ$ action on $F$.
 Thus
using the periodic monodromy argument in \cite{Milnor}, we get
\begin{Proposition}
The zeta function of the monodromy $h:F\to F$ is given by
\[
 \zeta(t)=(1-t^{q})^{-\chi(F)/q}.
\]In particular, if $q=1$,
$h=\id_F$ and 
$\zeta(t)=(1-t)^{-\chi(F)}$.
\end{Proposition}
\subsection{Projective mixed Curves }
Let   $C$ be  a smooth $C^\infty$  surface  embedded  in $\BP^2$
and let  $g$ be the genus of $C$ and let $q$ be the  embedding degree
of $C$.
It is known that the following inequality is satisfied.
\[
 g\ge \frac{(q-1)(q-2)}2.
\]
This was first conjectured by  R.  Thom and it has been proved by
many people.
For example see Kronheimer-Mrowka, \cite{Kronheimer-Mrowka}.
We are interested to present $C$ as a mixed algebraic curve
in the smallest embedding  degree $q$ of  a Riemann surface of
a given
genus $g$ as a mixed algebraic curve. (So we are not interested in the
embedding with $q=0$.)
In our previous paper,  we have used the join type construction starting
from a strongly  polar homogeneous polynomial of two variables
$f(z_1, z_2,\bar z_1,\bar z_2)$ of polar degree $q$ and radial degree $q+2r$
and we considered 
\[
 g(z_1,z_2, z_3,\bar z_1,\bar z_2, \bar z_3)=f(z_1, z_2,\bar z_1,\bar
 z_2)+
z_3^{q+r}\bar z_3^r.
\]
Using such a polynomial,
we have shown that there exists a mixed curve
of a given genus  $g$ with the embedding degree $2$ (\cite{MC}).
Note that if degree $q=1$, the join theorem (\cite{Molina})
says that the Euler number of the Milnor fiber of $g$ is 1
(i.e., the Milnor number is 0) and thus we only get genus 0.
Thus to get a mixed curve of polar degree $1$ and the genus arbitrary
large,
we have to find another type of polynomials.
This is the reason we introduce {\em polar weighted homogeneous
polynomials
of twisted join type} (See \S 3). For example, in the above setting, we consider
the polynomial:
\[
  g'(z_1,z_2, z_3,\bar z_1,\bar z_2, \bar z_3)=f(z_1, z_2,\bar z_1,\bar
 z_2)+
\bar z_2 z_3^{q+r}\bar z_3^{r-1}.
\]
 Using  polynomials of this type, we will show that
{\em there exists  a mixed surface with the polar degree $q=1$
for any $g$ (Theorem \ref{main2}, Corollary \ref{main3}).}

This paper is a continuation of our previous papers \cite{OkaMix,OkaBrieskorn,MC} and we use the same 
notations as those we have used previously.
\section{Mixed projective curves}
Let $ \cM(q+2r,q;n)$ be the space of strongly polar
  homogeneous polynomials
of n-variables $z_1,\dots, z_n$
with polar degree $q$ and radial degree $q+2r$.
\subsection{Important mixed affine curves}
We consider the  following mixed strongly polar homogeneous polynomial
of two variables:
\[
 h_{q,r,j}(\bfw,\bar \bfw)=(w_1^{q+j}\bar w_1^j+w_2^{q+j}\bar
   w_2^{j})(w_1^{r-j}-\al w_2^{r-j})(\bar w_1^{r-j}-\be \bar
   w_2^{r-j}),\quad r\ge j\ge 0
\]
with $\al,\be\in \BC^*$ generic. This polynomial plays a key role for
the construction.
Note that $h_{q,r,j}$ is a strongly  polar homogeneous polynomial with radial degree
$q+2r$ and
polar degree $q$ respectively i.e., $h_{q,r,j}\in \cM(q+2r,q;2)$.
Then the Milnor fiber $H_{{q,r,j}}:=h_{q,r,j}\inv(1)$ of $h_{q,r,j}$ is connected.
The Euler characteristic of $\chi(H_{q,r,j}^*)$ ( where $H_{{q,r,j}}^*=H_{{q,r,j}}\cap \BC^{*2}$)
is  given by 
\[
 \chi(H_{{q,r,j}}^*)=-r_{q,r,j}\,\times \,q\quad 
\text{and}\,\,\,\,  \chi(H_{{q,r,j}})=-r_{q,r,j}\,q + 2q
\]
where $r_{q,r,j}$ is
the
 link component number
of the mixed curve $C=h_{q,r,j}\inv(0)$.
 Note that  the link component number
$ r_{q,r,j}$ is given by
$ r_{q,r,j}=q+2(r-j)$
by Lemma 64, \cite{OkaMix}.
Thus
\begin{Proposition}\label{H(q,r,j)}
 \[
  \chi(H_{{q,r,j}})\,=\,-q\,((q-2)\,+\,2\,(r-j))
 \]
 \end{Proposition}
\subsection{Join type polynomials} We consider the following
 strongly polar homogeneous
polynomial of join type.
\[\begin{split}
&f_{q,r,j}(\bfz,\bar\bfz)= h_{q,r,j}(\bfw,\bar\bfw)+z_3^{q+r}\bar z_3^r,
\quad \bfw=(z_1,z_2)\\
      \end{split}
\] 
The the Milnor fiber $F_{{q,r,j}}=f_{q,r,j}\inv(1)$ of $f_{q,r,j}$ is connected.
By the Join theorem ( Cisneros-Molina \cite{Molina}),
$F_{q,r,j}$  is a simply connected 2-dimensional CW-complex so that 
\[\begin{split}
 \chi(F_{q,r,j})&=-(q-1)\chi(H_{q,r,j})+ q\\
&=q(q-1)(q-2)+2q(q-1)(r-j)+q.
\end{split}
\]
Let $C_{q,r,j}$ be the projective curve defined by
$\{f_{q,r,j}(\bfz,\bar\bfz)=0\}$ in $\BP^2$. By Corollary \ref{genus formula1},
the genus $g(C_{q,r,j})$ of $C_{q,r,j}$ is given by
\[
 g(C_{q,r,j})=\frac{(q-1)(q-2)}2+(q-1)(r-j)\ge
\frac{(q-1)(q-2)}2.
\]
For $q=2$, we get
\[
 g(C_{2,r,j})=(r-j)\ge 0.
\]
Thus this shows that for arbitrary $g\ge 0$, the mixed curve $C_{2,g+j,j}$
is a curve of genus $g$ and the embedding degree $2$.
Note that $g(C_{1,r,j})=0$. Thus $q=1$ gives only  rational curves.
Therefore to get a mixed  curve with the embedding degree $1$, the join type polynomials
can not be used.
\section{Twisted join type polynomial}
In this section, we introduce a new class of mixed polar weighted polynomials
which  we use  to construct  curves with embedded degree 1.
Let $f(\bfz,\bar\bfz)$ be a polar weighted homogeneous polynomial
of $n$-variables $\bfz=(z_1,\dots, z_n)$.
Let $Q={}^t(q_1,\dots, q_n),\,P={}^t(p_1,\dots,p_n)$
be the radial and polar weight respectively  and 
let $d,\,q$ be the radial and polar degree respectively.
For simplicity, we call that
$Q'={}^t(q_1/d,\dots, q_n/d)$ and $P'={}^t(p_1/q,\dots, p_n/q)$
{\em the normalized radial weights} and {\em the normalized polar
weights}
respectively.
Consider the mixed polynomial of $(n+1)$-variables:
\[
 g(\bfz,\bar\bfz,w,\bar w)\,=\, f(\bfz,\bar\bfz)+
\bar z_n w^a{\bar w}^b,\quad a>b.
\]
Consider the rational numbers $\bar q_{n+1},\,\bar p_{n+1}$
 satisfying
\[
\frac {q_n}{d}+(a+b)\,\bar q_{n+1}=1,\quad -\frac{p_n}{q}+(a-b)\, \bar p_{n+1}=1.
\]
We assume that $q_n<d$ so that $\bar q_{n+1},\bar p_{n+1}$  are
positive rational numbers.
The polynomial $g$ is a polar weighted homogeneous polynomial with 
the normalized radial and polar weights 
$\widetilde{Q'}={}^t(q_1/d,\dots, q_n/d,\bar q_{n+1})$ and
$\widetilde{P'}={}^t(p_1/q,\dots, p_n/q,\bar p_{n+1})$ respectively.
The radial and polar degree of $g$
are  given by \lcm$(d,denom(\bar q_{n+1}))$ and \lcm$(q,denom(\bar p_{n+1}))$
where $denom(x)$ is the denominator of $x\in \BQ$.
We call $g$ {\em a twisted join  of $f(\bfz,\bar \bfz)$
and $\bar z_n w^a{\bar w}^b$}. 
We say that $g$ is a polar weighted homogeneous polynomial of
{\em  twisted join type}. A twsited join type polynomial behaves
differently than the simple join type, as we will see below.

We recall that 
$f(\bfz,\bar \bfz)$ is  called to be {\em 1-convenient} if
the restriction of $f$ to each coordinate hyperplane
$f_i:=f|_{\{z_i=0\}}$ is non-trivial for  $i=1,\dots,n$ (\cite{OkaPolar})
\begin{Lemma}\label{key1} Assume that $n\ge 2$ and $f$ is 1-convenient.
 Then 
$$\phi_{\sharp}:\,\pi_1((\BC^*)^n\setminus F_f^*)\cong \BZ^{n}\times
 \BZ$$
is an isomorphiam
where $\phi$ is  the canonical mapping
 $\phi: (\BC^*)^n\setminus F_f^*\to (\BC^*)^n\times (\BC\setminus\{1\})$
 defined by 
 $\phi(\bfz)=(\bfz,f(\bfz,\bar\bfz))$
and $F_f^*:=f\inv(1)\cap (\BC^*)^n$.
 \end{Lemma}
 \begin{proof}
Let us use the notations:
 \[
   D_\de:=\{\eta\in \BC||\eta|\le \de\},\quad S_\de(1)=\{\eta\in
  \BC||\eta-1|=\de\}.
  \]
Denote by $\hat f$ the restriction of $f$ to $(\BC^*)^n$.
 The fact that  the mapping 
$\hat f: (\BC^*)^n\setminus f\inv(0)\to
  \BC^*$ is a fibration and 
the inclusion $D_{1-\eps}\cup S^1_{\eps}\hookrightarrow \BC\setminus\{1\}$
is a  deformation retract  implies  the following inclusion is also a deformation retract:
  \[
\iota:\,  {\hat f}\inv(D_{{1-\eps}})\cup {\hat f}\inv(S_{\eps}(1))\subset
  (\BC^*)^n\setminus F_f^*,\quad 0<\eps\ll 1.
  \]
%
  On the other hand, ${\hat f}\inv(S_\eps(1))\cong f\inv(1-\eps)\times
  S^1_\eps(1)\cong F_f^*\times S^1_{\eps}$ and
  $\pi_1(f\inv(S_\eps(1)))\cong \pi_1(F_f^*)\times \BZ$.
The 1-convenience of $f$
implies the homomorphism
$i_{\sharp}:\pi_1(F_f^*)\to \pi_1((\BC^*)^n)$ is surjective.
  Moreover $f\inv(D_{1-\eps})$ is homotopic to $(\BC^*)^n$,
as $D_{1-\eps}\hookrightarrow \BC$ is a deformation retract. Thus
 the
  assertion follows from the van Kampen lemma, applied 
to the decomposition
\[\begin{split}
 {\hat f}\inv(D_{1-\eps}\cup S^1_{\eps}(1))&={\hat f}\inv(D_{{1-\eps}})\cup {\hat f}\inv(S_{\eps}(1)),\\
{\hat f}\inv(D_{{1-\eps}})\cap {\hat f}\inv(S_{\eps}(1))&={\hat
   f}\inv(1-\eps)\cong F_f^*.
\end{split}
\]
  \end{proof}
Put $F_{f_n}:=f_n\inv(1)=F_f\cap\{z_n=0\}\subset \BC^{n-1}$
with $f_n:=f|_{\BC^n\cap\{z_n=0\}}$.
  \begin{Theorem} \label{main1}Assume that $n\ge 2$ and  $f$ is 1-convenient
  and 
 $g(\bfz,\bar\bfz,w,\bar w)$ is a twisted join polynomial as above.
 Then \begin{enumerate}
       \item
            the Milnor fiber of $g$, $F_g=g\inv(1)$, is simply connected.
            \item
            The Euler characteristic of $F_g$ is given by the formula:
            \[
             \chi(F_g)=-(a-b-1)\chi(F_f)+(a-b)\chi(F_{f_n}).
                       \]
\end{enumerate}
  \end{Theorem}

\begin{proof}Consider $F_g^*:=F_g\cap (\BC^*)^{n+1}$
and  the projection map
 $\pi:\,F_g^*\to (\BC^*)^n$ defined by $(\bfz,w)\mapsto \bfz$.
 Then the image of  $F_g^*$ by $\pi$ is $(\BC^*)^n\setminus F_f^*$ and
 $\pi:F_g^*\to (\BC^*)^n\setminus F_f^*$
 gives an $(a-b)$-cyclic covering.
In fact the fiber $\pi\inv(\bfz)$
is given as 
\[
 \pi\inv(\bfz)=\{(\bfz,w)\,|\, w^a\bar w^b=\frac{1-f(\bfz,\bar\bfz)}{\bar z_n}\}
\]
Therefore
\[
 \pi_1((\BC^*)^n\setminus
 F_f^*)/\pi_{\sharp}(\pi_1(F_g^*))\cong \BZ/(a-b)\BZ.
\]
 By Lemma \ref{key1},
 we see that  $\pi_1((\BC^*)^n\setminus F_f^*)\cong \BZ^{n+1}$
 and  any subgroup of $\BZ^{n+1}$ with a finite index is a free abelian
 group of the same rank $n+1$. Therefore
 $\pi_1(F_g^*)\cong \BZ^{n+1}$. Note that $g(\bfz,\bar \bfz,w,\bar w)$
is 1-convenient. Thus taking normal slice of each smooth  divisor
 $z_i=0$
in $F_g$,  we see that 
\[
 \iota_{\sharp}:\pi_1(F_g^*)\to
 \pi_1((\BC^*)^{n+1})
\]
is  surjective.
Consider the inclusion map $\iota: F_g^*\to (\BC^*)^{n+1}$.
 If  $\iota_{\sharp}$ is not injective,
$\pi_1((\BC^*)^{n+1})\cong \pi_1(F_g^*)/\Ker\,\iota_{\sharp}$
can not be a free abelian group of rank $n+1$.
Thus $\iota_{\sharp}:\pi_1(F_g^*)\to \pi_1((\BC^*)^{n+1})$
is an isomorphism.

 For the proof of the assertion (2), we apply the additivity of the Euler
 characteristic
 to the union $F_g=F_g^{*\{n\}}\cup F_{g_n} $
 where $F_g^{{*\{n\}}}:=F_g\cap \{z_n\ne 0\}$ and $F_{g_n}:=F_g\cap
 \{z_n=0\}$.
Note that $F_{g_n}\cong F_{f_n}\times \BC$.
 Put $\BC^{{*\{n\}}}=\BC^n\cap\{z_n\ne 0\}$ and $F_f^{{*\{n\}}}=F_f\cap
 \{z_n\ne 0\}$.
In the following, we consider the projection
$\pi_n:\,\BC^{n+1}\to \BC^n$ defined by $\pi_n(\bfz,w)=\bfz$.
Note that $\pi_n\inv(F_f)=F_f\times \BC$
and $F_g^{{*\{n\}}}\cap\pi_n\inv( F_f)=\{(\bfz,0)\,|\, \bfz\in F_f^{{*\{n\}}} \}$.
 \[
  \begin{split}
 \chi(F_g^{{*\{n\}}})&=\chi(F_g^{{*\{n\}}}\setminus\pi_n\inv( F_f))+ \chi(F_g^{{*\{n\}}}\cap\pi_n\inv( F_f))\\
   &=(a-b)\chi(\BC^{{*\{n\}}}\setminus F_{f}^{{*\{n\}}})+\chi(F_f^{{*\{n\}}})\\
   &=\,-(a-b-1)\chi(F_f^{{*\{n\}}})\\
   \chi(F_{g_n})&=\chi(F_{f_n}\times \BC)=\chi(F_{f_n}).
   \end{split}
 \]
The last equality follows from 
$F_{g_n}=F_{f_n}\times \BC$.
 To complete the proof, we use the additivity of the Euler
 characteristic which gives the  equality
 \[
  \chi(F_f)=\chi(F_f^{{*\{n\}}})+\chi(F_{f_n}).
 \]
\end{proof}
\subsection{Construction of a family of mixed curves with polar  degree q}
Now we are ready to construct a key family of mixed curves with embedding degree q.
Recall the polynomial:
\[
 h_{q,r,j}(\bfw,\bar \bfw):=(z_1^{q+j}\bar z_1^j+z_2^{q+j}\bar
   z_2^{j})(z_1^{r-j}-\al z_2^{r-j})(\bar z_1^{r-j}-\be \bar
   z_2^{r-j}),\quad \bfw=(z_1,z_2).
\]
$h_{q,r,j}(\bfw,\bar\bfw)$ is 1-convenient strongly polar homogeneous
polynomial
with the   radial degree $q+r$ and the  polar degree $ q$ respectively.
The constants $\al,\be$ are generic. For this, it suffices to assume that
$|\al|,|\be|\ne 0,1$ and $|\al|\ne |\be|$.
Consider the twisted join polynomial of 3 variables $z_1,z_2,z_3$:
\[
 s_{q,r,j}(\bfz,\bar \bfz)=h_{q,r,j}(\bfw,\bar\bfw)+\bar z_2
 z_3^{q+r}\bar z_3^{r-1},\quad \bfz=(z_1,z_2,z_3).
\]
Let $F_{q,r,j}=s_{q,r,j}\inv(1)\subset \BC^3$ be the Milnor fiber and let $S_{q,r,j}\subset \BP^2$ be the
corresponding mixed projective curve:
\[
 S_{q,r,j}=\{[\bfz]\in \BP^2\,|\, s_{q,r,j}(\bfz,\bar \bfz)=0\}.
\]
Note that $S_{q,r,j}$ is a smooth mixed curve.
The following describes  the topology of  $F_{q,r,j}$ and $S_{q,r,j}$.
\begin{Theorem}\label{main2}
 \begin{enumerate}
  \item The Euler characteristic of the Milnor fiber
$F_{q,r,j}$ is given by:
       \[
        \chi(F_{q,r,j})=q(q^2-q+1+2(r-j)).\]
        \item The genus of $S_{q,r,j}$ is given by:
        \[
        g(S_{q,r,j})=\frac{q(q-1)}2+(r-j)
        \]
       \end{enumerate}
 \end{Theorem}
\begin{proof}
Let $H_{q,r,j}=h_{q,r,j}\inv(1)$. Then by
Proposition \ref{H(q,r,j)},
\[
 \begin{split}
\chi(H_{q,r,j})&=-q(q-2+2(r-j))\\
\chi(H_{q,r,j}\cap\{z_2=0\})&=q
\end{split}
\]
and the assertion follows from Theorem \ref{main1}.
\end{proof}
\subsection{Mixed curves with polar degree 1} We  consider the case $q=1,j=0$:
 \[\begin{cases}
  h(\bfw,\bar\bfw)&:=(z_1+z_2)(z_1^r-\al z_2^r)(\bar z_1^r-\be \bar
    z_2^r)\\
f_r(\bfz,\bar\bfz)&:=h(\bfw,\bar\bfw)+\bar z_2 z_3^{r+1}\bar z_3^{r-1}\\
S_r&:=\{[\bfz]\in \BP^2\,|\,f_r(\bfz,\bar\bfz)=0\}.
\end{cases}
 \]
\begin{Corollary}\label{main3}
Let $S_r$ be the mixed curve as above. Then the  degree of $S_r$ is 1
and the genus of $S_r$ is $r$.
\end{Corollary}
\begin{proof} Let $F_r=f_r\inv(1)$ be the Milnor fiber of $f_r$.
By Theorem \ref{main1},
we have
$\chi(F_r)=2r+1$. Thus by Corollary \ref{genus formula1}, the assertion
 follows immediately.
\end{proof}
\begin{Remark}$h(\bfw,\bar\bfw)$ can be replaced by
$(z_1^{r+1}-z_2^{r+1})(\bar z_1-\be \bar z_2^r)$ without changing the
topology.
\end{Remark}
 \section{Further embeddings of smooth curves}
Consider a smooth curve $C\subset \BP^2$ with genus $g$. 
 If $C$ is a complex algebraic curve of degree $q$,
they are related by the Pl\"ucker formula
$g=\frac{(q-1)(q-2)}2$.
In particular, $q$ is the positive integer root of 
$x^2-3x+2-2g=0$. Thus for a given $g\ge 1$,  $q$
is unique if it exists.
 In this section, we consider this problem in the category of mixed
 projective curves. Consider the family of mixed curves.
 \[
    S_{q,r,1}:\, h_{q,r,1}(\bfw,\bar\bfw)+\bar z_2z_3^{q+r}\bar z_3^{r-1} 
   \]
We have shown that the genus $g$ is given as follows.
\[
 g=\frac {q(q-1)}2 + r-1.
\]
Assume that $g$  is fixed and we consider the possible degree $q$.
We can solve as
\[
 r=g-\frac{q(q-1)}2 +1.
\]
This shows that
\begin{Theorem}
For a given $g>0$ and $q$
which satisfies the inequality
\[
 g\ge \frac {q(q-1)}2,
\] the mixed curve
$S_{q,r,1}$ with $r=g-\frac{q(q-1)}2 +1$  has genus $g$ and degree $q$.
\end{Theorem}

\begin{Remark}Assume that 
\[
 (\sharp)\quad
\frac{q(q-1)}2\ge g\ge \frac {(q-1)(q-2)}2.
\]
For the construction of a curve with $\{g,q\}$ satisfying $(\sharp)$, we can
not use the surface $S_{q,r,1}$.
If $g-\frac{(q-1)(q-2)}2\equiv 0$ mod $q-1$, we can use the mixed curve
$C_{q,r,1}$. If $g\not \equiv \frac{(q-1)(q-2)}2$ mod $q-1$, we do not know if such
an embedding exists.
\end{Remark}
\section{Mixed polar weighted polynomial with polar degree 1 of $n$ variables}
Let us consider
 mixed polar weighted homogeneous 
polynomials of $n$ variables with polar  degree 1. They 
 have the following strong property:
\begin{Theorem}
Let $f(\bfz,\bar\bfz)$ be a polar weighted homogeneous polynomial of
 degree 1
of radial weight $(q_1,\dots, q_n; d)$ and polar weight $(p_1,\dots,p_n;1)$.
Then the Milnor fibration
$\vphi=f/|f|: \,S^{2n-1}\setminus K\to S^1$ with $K=f\inv(0)\cap S^{2n-1}$
is trivial.
In fact, the explicit diffeomorphism is given using the one-parameter 
family of diffeomorphisms of the monodromy flows
$h_\theta\,:\,F\to F_\theta$ with $\theta\in \BR$ and
 $F_\theta:=\vphi\inv(\exp (i\theta))$ and
\[
 h_\theta(\bfz)=\exp(i\theta)\circ \bfz
\]
where $\rho\circ \bfz=(\rho^{p_1}z_1,\dots, \rho^{p_n}z_n)$
and $\rho\in S^1$. Note that $h_{2\pi}=\id$.
The trivialization of the fibration is given by the diffeomorphism
$\psi:   F\times S^1\to S^{2n-1}\setminus K $ which is defined  by
\[
 \psi(\bfz,\exp(i\theta))=h_\theta(\bfz)
\]
\end{Theorem}
Observe that the trivialization is not an extension of the 
trivialization of the normal bundle of $K$ in $S^{2n-1}$.
\begin{Corollary}
Let $f(\bfw)$, $\bfw=(z_1,z_2)$ be a polar weighted homogeneous polynomial with polar
 degree 1. Then the link $K:=f\inv(0)\cap S^3$
is trivially  fibered over the circle.
Thus we have
\[
 \pi_1(S^3\setminus K)\cong \BZ\times \pi_1(F)
\]
where $F$ is the Milnor fiber.
\end{Corollary}
Let $f(\bfz,\bar\bfz)$ be a polar weighted  homogeneous polynomial
of $n$ variables.
On the topology of the hypersurface $F=f\inv(1)$, we propose the following basic
question.

\noindent
{\em Is the homological (or homotopical) dimension of $F$ is  $n-1$ 
under a certain condition (say mixed non-degeneracy)?
}

\vspace{.3cm} We say that $f(\bfz,\bar \bfz)$ satisfies {\em the
homological dimension property} if the assertion is satisfied for $F=f\inv(1)$.
There are several cases in which the assertion  is true.
\begin{enumerate}
\item
Simplicial type: Assume that  $f(\bfz,\bar\bfz)$ is a simplicial type 
polar weighted homogeneous polynomial. Then the homological dimension of
$F$ is at most $n-1$. This follows from Theorem
10, \cite{OkaPolar}.
\item (Join type) Assume that $f(\bfz,\bar \bfz)=h(\bfw,\bar \bfw)+k(\bfu,\bar\bfu)$
where 
$\bfw=(w_1,\dots, w_m)$,
$\bfu=(u_1,\dots,u_\ell)$ and $\bfz=(\bfw,\bfu)$. 
Assume that $h(\bfw,\bar \bfw)$, $k(\bfu,\bar\bfu)$ 
are polar weighted homogeneous polynomials which satisfies the
     homological dimension property. Then $f$ also satisfies the
     property.
This follows from the Join theorem by Cisneros Molino \cite{Molina}.
\end{enumerate}
\def\cprime{$'$} \def\cprime{$'$} \def\cprime{$'$} \def\cprime{$'$}
  \def\cprime{$'$} \def\cprime{$'$} \def\cprime{$'$} \def\cprime{$'$}


\begin{thebibliography}{1}

\bibitem{Molina}
J.~L. Cisneros-Molina.
\newblock Join theorem for polar weighted homogeneous singularities.
\newblock In {\em Singularities {II}}, volume 475 of {\em Contemp. Math.},
  pages 43--59. Amer. Math. Soc., Providence, RI, 2008.

\bibitem{Kronheimer-Mrowka}
P.~B. Kronheimer and T.~S. Mrowka.
\newblock The genus of embedded surfaces in the projective plane.
\newblock {\em Math. Res. Lett.}, 1(6):797--808, 1994.

\bibitem{Milnor}
J.~Milnor.
\newblock {\em Singular points of complex hypersurfaces}.
\newblock Annals of Mathematics Studies, No. 61. Princeton University Press,
  Princeton, N.J., 1968.

\bibitem{OkaPolar}
M.~Oka.
\newblock Topology of polar weighted homogeneous hypersurafces.
\newblock {\em Kodai Math. J.}, 31(2):163--182, 2008.

\bibitem{OkaBrieskorn}
M.~Oka.
\newblock On mixed {B}rieskorn variety.
\newblock {\em ArXiv 0909.4605v2}, XX(X), 2009.

\bibitem{MC}
M.~Oka.
\newblock On mixed projective curves.
\newblock {\em ArXiv 0910.2523}, XX(X), 2009.

\bibitem{OkaMix}
M.~Oka.
\newblock Non-degenerate mixed functions.
\newblock {\em Kodai Math. J.}, 33(1):1--62, 2010.

\end{thebibliography}
\end{document}